\documentclass{amsart}
\usepackage{amsfonts}
\usepackage{amsthm}
\usepackage{geometry}
\usepackage{amsmath}
\usepackage{txfonts}

\usepackage{url}
\usepackage{graphicx}
\usepackage{wrapfig}
\usepackage{ascmac}
\usepackage{bm}
\usepackage{indentfirst}
\usepackage{tikz}
\usepackage{color}
\usepackage{ulem}
\usepackage{comment}

\usetikzlibrary{matrix,calc,arrows,decorations.markings}

\newtheorem{thm}{Theorem}[section]
\newtheorem{rem}[thm]{Remark}
\newtheorem{defi}[thm]{Definition}
\newtheorem{ex}[thm]{Example}

\newtheorem{lem}[thm]{Lemma}
\newtheorem{prop}[thm]{Proposition}

\def\NZQ{\Bbb}

\def\RR{{\NZQ R}}
\def\CC{{\NZQ C}}

\def\SS{{\NZQ S}}
\def\TT{{\NZQ T}}

\def\ml{\mathcal{C}}
\def\ml1{\mathcal{C}^1}
\def\mlb1{\mathcal{C}_{b}^{1}}

\def\frk{\frak}

\def\Phi{{\frk n}}

\def\rank{{\rm rank}}

\def\A{{\mathcal A}}
\def\B{{\mathcal B}}

\newcommand{\Zt}[0]{\mathcal Z}

\usepackage{color}
\usepackage{ifthen}

\newboolean{usecolor}
\setboolean{usecolor}{true}

\ifthenelse{\boolean{usecolor}}
{
  \definecolor{colore}{cmyk}{0,1,0.6,0}
  \definecolor{coloregen}{cmyk}{0.7,0,1,0}
  \definecolor{coloresimo}{cmyk}{1,0.6,0,0}
}
{
  \definecolor{colore}{cmyk}{0,0,0,1}
  \definecolor{coloregen}{cmyk}{0,0,0,1}
  \definecolor{coloresimo}{cmyk}{0,0,0,1}
}

\title{A linear condition for non-very generic discriminantal arrangements}
\author{Simona Settepanella}
\author{So Yamagata}
\address[1]{Department of Economics and Statistics, Torino University, Torino, Italy}
\email{simona.settepanella@unito.it}
\address[2]{%
Department of Mathematics,
Hokkaido University, Sapporo, Japan.}
\email{so.yamagata@math.sci.hokudai.ac.jp}

\thanks{The second author was supported by the Program for Leading Graduate Schools (Hokkaido University Ambitious
Leader's Program) and JSPS Research Fellowship for Young Scientists Grant Number 20J10012.}

\begin{document}

\maketitle

\begin{abstract}
The discriminantal arrangement is the space of configurations of $n$ hyperplanes in generic position in a $k$ dimensional space (see \cite{MS}). Differently from the case $k=1$ in which it corresponds to the well known braid arrangement, the discriminantal arrangement in the case $k>1$ has a combinatorics which depends from the choice of the original $n$ hyperplanes. It is known that this combinatorics is constant in an open Zarisky set $\Zt$, but to assess weather or not $n$ fixed hyperplanes in generic position belongs to $\Zt$ proved to be a nontrivial problem. Even to simply provide examples of configurations not in $\Zt$ is still a difficult task. In this paper, moving from a recent result in \cite{SSc}, we define a \textit{weak linear independency} condition among sets of vectors which, if imposed, allows to build configurations of hyperplanes not in $\Zt$. We provide $3$ examples.
\end{abstract}

\section{Introduction}

In 1989 Manin and Schechtman (\cite{MS}) introduced the {\it discriminantal arrangement} $\B(n,k,\A^0)$ which hyperplanes consist of the non-generic parallel translates of the generic arrangement $\A^0$ of $n$ hyperplanes in $\CC^k$. $\B(n,k,\A^0)$ is a generalization of the braid arrangement (\cite{OT}) with which $\B(n,1)=\B(n,1,\A^0)$ coincides. \\
These arrangements, which have several beautiful relations with diverse problems (see, for instance Kapranov-Voevodsky \cite{KV3}, \cite{KV1},\cite{KV2}, the vanishing of cohomology of bundles on toric varieties \cite{Per}, the representations of higher braid groups \cite{Koh}) proved to be not only very interesting, but also quite tricky. Indeed, differently from the braid arrangement $\B(n,1)$, their combinatorics depends on the original arrangement $\A^0$ if $k>1$. This generated few misunderstanding over the years.\\
Manin and Schechtman introduced the discriminantal arrangement in order to model higher Bruhat orders. To do so, they only used the combinatorics of $\B(n,k,\A^0)$ when $\A^0$ varies in an open Zarisky set $\Zt$. In 1991 Ziegler showed ( see Theorem 4.1 in \cite{Zie} ) that their description was not complete and that a slightly different construction was needed. In particular in \cite{FZ} Felsner and Ziegler showed that, in the representable case, what was missed in Manin and Schechtman construction was the part coming from the case $\A^0 \notin \Zt$.\\
As pointed out by Falk in 1994 (\cite{Falk}), the construction in \cite{MS} led to the misunderstanding that the intersection lattice of  $\B(n,k,\A^0)$ was independent of the arrangement $\A^0$ ( see, for instance, \cite{Orl}\cite{OT}, or \cite{Law}). In \cite{Falk} Falk shows that this was not the case, providing an example of a combinatorics of $\B(6,3,\A^0)$ different from the one presented by Manin and Schechtman. Nevertheless Falk too was not aware that such an example already existed in \cite{Crapo}.\\
Actually, Crapo already introduced the discriminantal arrangement in 1985 with the name of \textit{geometry of circuits} (see \cite{Crapo}). In a more combinatorial setting, he defined it as the matroid $M(n,k, \mathcal{C})$ of circuits of the configuration $\mathcal{C}$ of $n$ generic points in $\RR^k$. The circuits of the matroid $M(n,k, \mathcal{C})$ are the hyperplanes of $\B(n,k,\A^0),$ when $\A^0$ is the  arrangement of the hyperplanes in $\RR^k$ orthogonal to the vectors joining the origin with the $n$ points in $\mathcal{C}$ ( for further development see \cite{CR} ). In this paper, Crapo provides also an example of an arrangement $\A^0$ of $6$ lines in the real plane for which the combinatorics of $\B(6,2,\A^0)$ in rank $3$ is different from the one described in \cite{MS}. \\
In 1997 Bayer and Brandt (see \cite{BB}) cast a light on this difference naming \textit{very generic} the arrangements $\A^0 \in \Zt$ ( the one simply called generic in \cite{MS}) and non very generic the others. They conjectured a description of the intersection lattice of $\B(n,k,\A^0), \A^0 \in \Zt$, subsequently proved by Athanasiadis in 1999 in \cite{Atha} (as far as it is known to the authors, this is the first paper on the literature on discriminantal arrangements in which there is a reference to the Crapo's paper \cite{Crapo}). It is worthy to notice that Bayer and Brandt mentioned in \cite{BB} that, in their opinion, a candidate for a very generic arrangement could have been the cyclic arrangement. In Subsection \ref{subse:simple}, by means of the Theorem 4.1 in \cite{SSY1}, we show that this is not the case since there are cyclic arrangements which are non-very generic.\\
Moreover even if Athanasiadis proved that the combinatorics of $\B(n,k,\A^0), \A^0 \in \Zt$ can be described by opportunely defined sets of subsets of $\{1,\ldots,n\}$, the contrary is not true in general. In Section \ref{sec:example}, by means of our main result, we provide  two examples of non-very generic arrangements $\A^0$ for which the combinatorics of $\B(n,k,\A^0)$ satisfies Bayer-Brandt-Athanasiadis numerical properties (see Examples \ref{ex:MS(16,11)} and  \ref{ex:MS(10,3)}).
We point out that to provide such examples is a nontrivial result since, until very recently, 1985 Crapo's and 1997 Falk's examples were, essentially, the only known examples of non-very generic arrangements.\\
Finally, in order to better understand the difficulty behind this problem, we remark that even with the Bayer-Brandt-Athanasiadis description, the enumerative problem of funding the characteristic polynomial of the discriminantal arrangement in the very generic case is nontrivial (see, for instance, \cite{NT}).\\
Recently, following a result in \cite{LS} which completely describes the rank $2$ combinatorics of $\B(n,k,\A^0)$ for any $\A^0$, the authors tried to better understand the combinatorics of the discriminantal arrangement (see \cite{SSY1}, \cite{SSY2}, \cite{SSc}) and its connection with special configurations of points in the space (see \cite{DPS}, \cite{SS}).  In particular in \cite{SSc} the authors generalize the dependency condition given in \cite{LS} providing a sufficient condition for the existence, in rank $r > 2$, of non-very generic intersections, i.e. intersections which do not appear in $\B(n,k,\A^0), \A^0 \in \Zt$.\\ 
In this paper, moving from the construction in \cite{SSc}, we introduced the notion of \textit{weak linear independency} among sets of vectors proving that if there are vectors in the hyperplanes of $\A^0$ which form weakly linearly independent sets, then $\A^0$ is non-very generic. This result allowed to build several examples of non-very generic arrangements in higher dimension (see also \cite{So}).\\  
The content of the paper is the following. In Section \ref{sec:prelim}, we recall the definition of discriminantal arrangement and, following \cite{SSc}, the definition of \textit{simple} intersection, $K_\TT$-translated and $K_\TT$-configuration. In Section \ref{sec:vect}, we introduce the notion of $K_\TT$-vector sets, weak linear independency and we prove our main result, Theorem \ref{thm:main2}. In the last section we provide three examples of non-very generic arrangements obtained by imposing the condition stated in Theorem \ref{thm:main2}.

\section{Preliminaries}\label{sec:prelim}
\subsection{Discriminantal arrangement} Let $\A^0 = \{ H_1^0, \dots, H_n^0 \}$ be a central arrangement in $\CC^k, k<n$ such that any $m$ hyperplanes intersect in codimension $m$ at any point except for the origin for any $m \leq k$. We will call such an arrangement a \textit{central generic\footnote{Here we use the word \textit{generic} to stress that $\A^0$ admits a translated which is a generic arrangement.} arrangement}. The space $\SS[\A^0]$ (or simply $\SS$ when dependence on $\A^0$ is clear or not essential) will denote the space of parallel translates of $\A^0$, that is the space of the arrangements $\A^t= \{ H_1^{x_1}, \dots, H_n^{x_n} \}$, $t = (x_1, \dots, x_n) \in \CC^n$, $H_i^{x_i} = H_i^0 + \alpha_i x_i$, $\alpha_i$ a vector normal to $H_i^0$. There is a natural identification of $\SS$ with the $n$-dimensional affine space $\CC^n$ such that the arrangement $\A^0$ corresponds to the origin. In particular, an ordering of hyperplanes in $\A^0$ determines the coordinate system in $\SS$ (see \cite{LS}). \\
The closed subset of $\SS$ formed by the translates of $\A^0$ which fail to form a generic arrangement is a union of hyperplanes $D_L \subset \SS$ (see \cite{MS}). Each hyperplane $D_L$ corresponds to a subset $L = \{ i_1, \dots, i_{k+1} \} \subset$  [$n$] $\coloneqq \{ 1, \dots, n \}$ and it consists of $n$-tuples of translates of hyperplanes $H_1^0, \dots, H_n^0$ in which translates of $H_{i_1}^0, \dots, H_{i_{k+1}}^0$ fail to form a generic arrangement. The arrangement $\B(n, k, \A)$ of hyperplanes $D_L$ is called $discriminantal$ $arrangement$ and has been introduced by Manin and Schechtman in \cite{MS}.\\
It is well known (see, among others \cite{Crapo},\cite{MS}) that there exists an open Zarisky set $\mathcal{Z}$ in the space of (central) generic arrangements of $n$ hyperplanes in $\CC^k$, such that the intersection lattice of the discriminantal arrangement $\mathcal{B}(n,k,\A)$ is independent from the choice of the arrangement $\A \in  \mathcal{Z}$. Accordingly to Bayer and Brandt in \cite{BB} we will call the arrangements $\A \in  \mathcal{Z}$ \textit{very generic} and \textit{non-very generic} the others.

\subsection{Simple intersections}\label{subse:simple} According to \cite{SSc} we call an element $X$ in the intersection lattice of the discriminantal arrangement $\mathcal{B}(n,k,\A)$ a \textbf{simple} intersection if $$X=\bigcap_{i=1}^r D_{L_i}, |L_i| = k+1, \mbox{ and } \bigcap_{i \in I}D_{L_i} \neq D_S, \mid S \mid > k+1 \mbox{ for any } I \subset [r], \mid I \mid \geq 2 \quad .$$ We call multiplicity of the simple intersection $X$ the number $r$ of hyperplanes intersecting in $X$. In \cite{SSc} authors proved that the following Proposition holds.
\begin{prop}\label{pro:main}If the intersection lattice of the discriminantal arrangement $\mathcal{B}(n,k,\A)$ contains a simple intersection of rank strictly less than its multiplicity, then $\A$ is non-very generic.
\end{prop}
\noindent
It is nontrivial to assess whether or not an arrangement is very generic. For example,
in \cite{BB} Bayer and Brandt guessed that the cyclic arrangement could have been a good candidate to build very generic arrangements. We can now show that this is not true in general. For instance, the cyclic arrangement $\A^0 \in \RR^3$ with hyperplanes normal to the vectors $\alpha_i=(1,t_i,t_i^2),(t_1,t_2,t_3,t_4,t_5,t_6)=(1,-1,a,-a,b,-b), a,b\neq 1, a\neq b$, is generic but non-very generic. Indeed the vectors $\alpha_1 \times \alpha_2$, $\alpha_3 \times \alpha_4$ and $\alpha_5 \times \alpha_6$ are linearly dependent and by Theorem 4.1 in \cite{SSY1} this is equivalent to the intersection $X=D_{\{1,2,3,4\}}\cap D_{\{1,2,5,6\}} \cap D_{\{3,4,5,6\}}$ be a simple intersection of multiplicity $3$ in rank $2$. By Proposition \ref{pro:main}, we get that $\A^0$ is non-very generic.

\subsection{$\mathbf{K_{\TT}}$-translated and $\mathbf{K_{\TT}}$-configurations}\label{sub:KT}
Fixed a set $\TT = \{ L_1, \dots, L_r \}$ of subsets $L_i \subset [n]$, $|L_i| = k+1$, for any arrangement $\A=\{H_1,\ldots, H_n\}$ translated of $\A^0$ we will denote by $P_i = \bigcap_{p \in L_i} H_p$ and $H_{i,j} = \bigcap_{p \in L_i \cap L_j} H_p$. Notice that $P_i$ is a point if and only if $\A \in D_{L_i}$, it is empty otherwise.
Following \cite{SSc} we will call the set $\TT$ an $r$-\textbf{set} if the conditions
 \begin{equation}\label{eq:proper1}
 \bigcup_{i=1}^r L_i = \bigcup_{i \in I} L_i  \quad \mbox{ and} \quad L_i \cap L_j \neq \emptyset
 \end{equation}
 are satisfied for any subset $I \subset [r], \mid I \mid=r-1$ and any two indices $1 \leq i < j \leq r$.\\
Given an $r$-set $\TT$ authors in \cite{SSc} defined:\\
\paragraph{$\mathbf{K_{\TT}}$-\textbf{translated}} A translated $\A=\{H_1, \ldots, H_n\}$ of $\A^0$ will be called $\mathbf{K_{\TT}}$ or $\mathbf{K_{\TT}}$-translated if each point
$P_i=\bigcap_{p \in L_i} H_p \neq \emptyset$ is the intersection of exactly the $k+1$ hyperplanes indexed in $L_i$ for any $L_i \in \TT$. 
\paragraph{$\mathbf{K_{\TT}}$-\textbf{configuration} $K_\TT(\A)$} Given a $\mathbf{K_{\TT}}$-translated $\A$, the complete graph having the points $P_i$ as vertices and  the vectors $P_iP_j$ as edges will be called \textbf{$K_{\TT}$-configuration} and denoted by $K_\TT(\A)$. \\

\section{An algebraic condition for non-very genericity}\label{sec:vect}
The discriminantal arrangement $\B(n,k,\A^0)$ is not essential arrangement of center $D_{[n]}=\bigcap_{L\subset n, \mid L \mid=k+1}D_L \simeq \CC^k$. The center is formed by all translated $\A^t$ of $\A^0$ which are central arrangements. If we consider its essentialization $ess(\B(n,k, \A^0))$ in $\CC^{n-k} \simeq \SS / D_{[n]}$, an element $\A^t \in ess(\B(n,k, \A^0))$ will corresponds uniquely to a translation $t \in \CC^n/ C \simeq \CC^{n-k}$, $C=\{t \in \CC^n \mid \A^t \mbox{ is central} \}$. The following proposition arises naturally.
\begin{prop}\label{defi:indepK_TT} Let $\A^0$ be a generic central arrangement of $n$ hyperplanes in $\CC^k$. Translations $\A^{t_1}, \dots, \A^{t_d}$ of $\A^0$ are linearly independent vectors in $\SS / D_{[n]}\simeq \CC^{n-k}$ if and only if $t_1,\ldots, t_d$ are linearly independent vectors in $\CC^n/C$. 
\end{prop}
\noindent
Given $\A^{t_1}, \dots, \A^{t_d}$ $K_{\TT}$-translated of $\A^0$, we will say that the $K_\TT$-configurations $K_\TT(\A^{t_i})$ are independent if $\A^{t_i}, i=1,\dots, d$ are.
\begin{center}
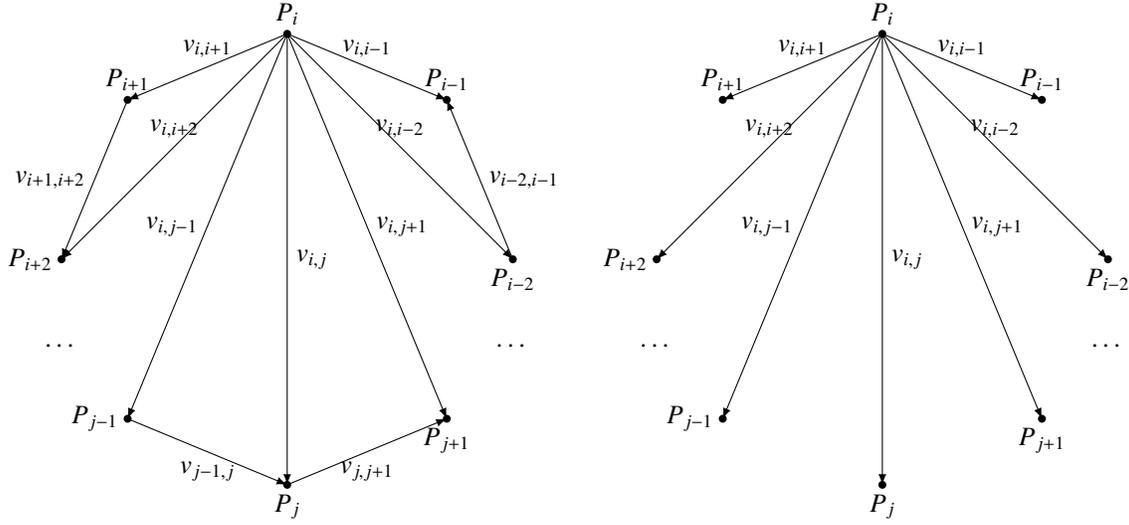
\begin{figure}[h]
\begin{tikzpicture}
\coordinate [label=above:$P_i$] (0) at (0,3);
\coordinate [label=above:$P_{i+1}$] (1) at ({-3/sqrt(2)},{3/sqrt(2)});
\coordinate [label=left:$P_{i+2}$] (2) at (-3,0);
\coordinate [label=below:$\dots$] (3) at (-3,-1);
\coordinate [label=left:$P_{j-1}$] (4) at ({-3/sqrt(2)},{-3/sqrt(2)});
\coordinate [label=below:$P_j$] (5) at (0,-3);
\coordinate [label=below:$P_{j+1}$] (6) at ({3/sqrt(2)},{-3/sqrt(2)});
\coordinate [label=below:$\dots$] (7) at (3,-1);
\coordinate [label=below:$P_{i-2}$] (8) at (3,0);
\coordinate [label=above:$P_{i-1}$] (9) at ({3/sqrt(2)},{3/sqrt(2)});

\begin{scope}
\draw[-latex] (0) -- node[above] {$v_{i,i+1}$} (1);
\draw[-latex] (0) -- node[above] {$v_{i,i+2}$} (2);
\draw[-latex] (0) -- node[left] {$v_{i,j-1}$} (4);
\draw[-latex] (0) -- node[right] {$v_{i,j+1}$} (6);
\draw[-latex] (0) -- node[above] {$v_{i,i-2}$} (8);
\draw[-latex] (1) -- node[left] {$v_{i+1,i+2}$} (2);
\draw[-latex] (4) -- node[below] {$v_{j-1,j}$} (5);
\draw[-latex] (5) -- node[below] {$v_{j,j+1}$} (6);
\draw[-latex] (8) -- node[right] {$v_{i-2,i-1}$} (9);
\draw[-latex] (0) -- node[above] {$v_{i,i-1}$} (9);
\draw[-latex] (0) -- node[right] {$v_{i,j}$} (5);

\fill (0) circle (1.5pt);
\fill (1) circle (1.5pt);
\fill (2) circle (1.5pt);
\fill (4) circle (1.5pt);
\fill (5) circle (1.5pt);
\fill (6) circle (1.5pt);
\fill (8) circle (1.5pt);
\fill (9) circle (1.5pt);

\end{scope}
\end{tikzpicture} \ \ \ 
\begin{tikzpicture}
\coordinate [label=above:$P_i$] (0) at (0,3);
\coordinate [label=above:$P_{i+1}$] (1) at ({-3/sqrt(2)},{3/sqrt(2)});
\coordinate [label=left:$P_{i+2}$] (2) at (-3,0);
\coordinate [label=below:$\dots$] (3) at (-3,-1);
\coordinate [label=left:$P_{j-1}$] (4) at ({-3/sqrt(2)},{-3/sqrt(2)});
\coordinate [label=below:$P_j$] (5) at (0,-3);
\coordinate [label=below:$P_{j+1}$] (6) at ({3/sqrt(2)},{-3/sqrt(2)});
\coordinate [label=below:$\dots$] (7) at (3,-1);
\coordinate [label=below:$P_{i-2}$] (8) at (3,0);
\coordinate [label=above:$P_{i-1}$] (9) at ({3/sqrt(2)},{3/sqrt(2)});

\begin{scope}
\draw[-latex] (0) -- node[above] {$v_{i,i+1}$} (1);
\draw[-latex] (0) -- node[above] {$v_{i,i+2}$} (2);
\draw[-latex] (0) -- node[left] {$v_{i,j-1}$} (4);
\draw[-latex] (0) -- node[right] {$v_{i,j+1}$} (6);
\draw[-latex] (0) -- node[above] {$v_{i,i-2}$} (8);
\draw[-latex] (0) -- node[above] {$v_{i,i-1}$} (9);
\draw[-latex] (0) -- node[right] {$v_{i,j}$} (5);

\fill (0) circle (1.5pt);
\fill (1) circle (1.5pt);
\fill (2) circle (1.5pt);
\fill (4) circle (1.5pt);
\fill (5) circle (1.5pt);
\fill (6) circle (1.5pt);
\fill (8) circle (1.5pt);
\fill (9) circle (1.5pt);
\end{scope}
\end{tikzpicture}
\caption{Diagonal vectors $v_{i,j}$ and their associated $K_\TT$-vector set.}\label{fig:K_T_vect}
\end{figure}
\end{center}
\subsection{$K_{\TT}$-vector sets}
Let $\A^t, t = (x_1, \dots, x_n)$ be a $K_\TT$-translated of $\A^0$ and $P_i^t$ denote the intersection $\bigcap_{p \in L_i} H^{x_p}_p$. Then to the $K_\TT$-configuration $K_\TT(\A^t)$ corresponds a unique family $\{v^t_{i,j}\}$ of vectors such that $P^t_i+v^t_{i,j}=P^t_j$. Notice that two different $K_\TT$-configurations can define the same family $\{v_{i,j}\}$\footnote{It is unique in the quotient space $\SS / D_{[n]}\simeq \CC^{n-k}$}. With the above notations, we provide the following definition.
\begin{defi}Let $\A^0$ be a central generic arrangement, $\TT$ an $r$-set, $\A^t$ a  $K_\TT$-translated of $\A^0$ and $i_0 \in [r]$ a fixed index. We call $K_\TT$-\textbf{vector set} the set of vectors $\{v^t_{i_0,j}\}_{j\neq i_0}$ which satisfies $P^t_{i_0}+v^t_{i_0,j}=P^t_{j}$ for any  $j\in [r], j \neq i_0$.
\end{defi}
\noindent
Since the vectors $\{v^t_{i,j}\}$ satisfy, by definition, the property $v^t_{k,l} = v^t_{i.l}- v^t_{i,k},$ then the set $\{v^t_{i,j}\}$ is uniquely determined by any $K_\TT$-vector set $\{v^t_{i_0,j}\}_{j\neq i_0}$ (see Figure \ref{fig:K_T_vect}). \\
Given a $K_\TT$-vector set we can naturally define the operation of multiplication by a scalar $$a\{v^t_{i_0,j}\}_{j\neq i_0} \coloneqq \{av^t_{i_0,j}\}_{j\neq i_0}, \quad a \in \CC$$ and sum of two different $K_\TT$-vector sets 
$$\{v^{t_1}_{i_0,j}\}_{j\neq i_0}+\{v^{t_2}_{i_0,j}\}_{j\neq i_0}=\{v^{t_1}_{i_0,j}+v^{t_2}_{i_0,j}\}_{j\neq i_0} \quad .$$
With the above notations and operations, we have the following definition. 

\begin{defi} For a fixed set $\TT$, $d$ different $K_\TT$-vector sets $\{ \{v^{t_h}_{i_0,j}\}_{j\neq i_0} \}_{h=1, \dots, d}$ are weakly linearly independent if and only if for any $a_1,\ldots,a_d \in \CC$ such that
\begin{equation}
\sum_{h=1}^{d} a_h \{v^{t_h}_{i_0,j}\}_{j\neq i_0} = 0,
\end{equation} 
then $a_1=\ldots=a_d=0$. 
\end{defi}
\noindent
The following remark is a key point to prove the connection between the linearly independence of $K_{\TT}$-configurations and the weakly linearly independence of the associated $K_\TT$-vector sets.
\begin{rem}\label{rem:corres} Let $K_\TT(\A^t)$ be the $K_\TT$-configuration of the arrangement $\A^t$, $K_\TT$-translated of $\A^0$. Then for any $c \in \CC$, the $K_\TT$-configuration $K_\TT(\A^{ct})$ is an "expansion" by $c$ of $K_\TT(\A^t)$, that is $v_{i,j}^{ct} = c v_{i,j}^{t}$. This is consequence of the fact that for any $i \in [r]$ the vector $OP^{ct}_i$ joining the origin with the points $P^{ct}_i$ satisfies $OP^{ct}_i=cOP^{t}_i$ by definition of translation. Hence $P^{ct}_iP^{ct}_j=cP^{t}_iP^{t}_j$, i.e. $$v_{i,j}^{ct} = c v_{i,j}^{t} \quad .$$ 
Analogously we have that, if $t_1,t_2 \in \CC^n$ are two translations then $$v_{i,j}^{t_1}+v_{i,j}^{t_2}=v_{i,j}^{t_1+t_2} \quad .$$
\end{rem}
\noindent
We can now prove the main lemma of this section.
\begin{lem}\label{lem:K_TTvec} Let $\A^0$ be a central generic arrangement of $n$ hyperplanes in $\CC^k$ and $\TT=\{L_1, \ldots, L_r\}$ be an $r$-set such that $[n]=\bigcup_{i=1}^r L_i$. The $K_{\TT}$-translated arrangements $\A^{t_1}, \dots, \A^{t_d}$  of  $\A^0$ are linearly independent if and only if their associated $K_\TT$-vector sets $\{ \{v^{t_h}_{i_0,j}\}_{j\neq i_0} \}_{h=1, \dots, d}$ are weakly linearly independent.
\end{lem}

\begin{proof}
By definition, $\A^{t_1}, \dots, \A^{t_d}$ are linearly independent if and only if the translations $t_1,\ldots,t_d$ are linearly independent vectors in $\CC^n/C$. Let's consider a linear combination $\sum_{h=1}^d a_h t_h$ of the vectors $t_h$ and the translated arrangements $\A^{a_ht_h}$. By Remark \ref{rem:corres} we have that the $K_{\TT}$-vector sets associated to $\A^{a_ht_h}$ satisfy the equalities:
$$
 \{v^{\sum_{h=1}^{d} a_ht_h}_{i_0,j}\}_{j\neq i_0}=\sum_{h=1}^{d} \{v^{a_h t_h}_{i_0,j}\}_{j\neq i_0}=\sum_{h=1}^{d} a_h \{v^{t_h}_{i_0,j}\}_{j\neq i_0} \quad .
$$
Hence $\sum_{h=1}^{d} a_h \{v^{t_h}_{i_0,j}\}_{j\neq i_0}=0$ if and only if $v^{\sum_{h=1}^{d} a_ht_h}_{i_0,j}=0$ for any $j$, that is $P_{i_0}^{\sum_{h=1}^{d} a_ht_h} \equiv P_j^{\sum_{h=1}^{d} a_ht_h}$. This is equivalent to $\A^{\sum_{h=1}^{d} a_ht_h}$ be a central arrangement of center $P_{i_0}^{\sum_{h=1}^{d} a_ht_h}$, i.e. $\sum_{h=1}^d a_h t_h \in C$ and the statement follows from Proposition \ref{defi:indepK_TT}.
\end{proof}
\noindent
The assumption that $\bigcup_{i=1}^r L_i=[n]$ in Lemma \ref{lem:K_TTvec}
is equivalent to consider a subset $\A'^0 \subset \A^0$ which only contains the hyperplanes indexed in the $\bigcup_{i=1}^r L_i \subset [n]$ in the more general case. Indeed if a (central) generic arrangement $\A^0$ contains a subarrangement $\A'^0$ which is non-very generic then $\A^0$ is obviously non-very generic. Analogously, if there exists a restriction arrangement $\A^{Y_{\A'}} = \{ H \cap Y_{\A'} | H \in \A^0 \setminus \A'\}, Y_{\A'}=\bigcap_{H \in \A'} H$ of $\A^0$ which is non-very generic, then $\A^0$ is non-very generic. The following main theorem of this Section follows.

\begin{thm}\label{thm:main2} Let $\A^0$ be a central generic arrangement of $n$ hyperplanes in $\CC^k$. If there exists an $r$-set $\TT=\{L_1, \ldots ,L_r\}$ with $\mid \bigcup_{i=1}^r L_i \mid=m$ and $\rank \bigcap_{p \in \bigcap_{i=1}^r L_i} H_p=y$, which admits $m-y-k-r'$ weakly linearly independent $K_{\TT}$-vector sets for some $r'<r$, then $\A^0$ is non-very generic. 
\end{thm}

\begin{proof}Let's consider the subarrangement $\A'$ of $\A^0$ given by the hyperplanes indexed in the $\bigcup_{i=1}^r L_i$ and its essentialization, i.e. the restriction arrangement $\A'^Y$, $Y=\bigcap_{p \in \bigcap_{i=1}^r L_i} H_p$. If $y=\rank~Y$ then the arrangement $\A'^Y$ is a central essential arrangement in $\CC^{m-y}, m=\mid \bigcup_{i=1}^r L_i \mid$. By Lemma \ref{lem:K_TTvec}, if $\A^{t_1},\ldots,\A^{t_{m-y-k-r'}}$ are $K_{\TT}$-translated of $\A'^Y$ associated to the $m-y-k-r'$ independent $K_{\TT}$-vector sets, then $\A^{t_1},\ldots,\A^{t_{m-y-k-r'}}$ are linearly independent vectors in $\SS[\A'^Y] / D_{[m]}\simeq \CC^{m-y-k}$. That is $\A^{t_1},\ldots,\A^{t_{m-y-k-r'}}$ span a subspace of dimension $m-y-k-r'$. On the other hand, by construction, $\A^{t_j}$ are $K_{\TT}$-translated, i.e.  $\A^{t_j}\in ess(X), X=\bigcap_{i=1}^r D_{L_i}$ for any $j=1, \ldots , m-y-k-r'$, that is the space spanned by $\A^{t_1},\ldots,\A^{t_{m-y-k-r'}}$ is included in $ess(X)$. This implies that the simple intersection $ess(X)$ has dimension $d \geq m-y-k-r' > m-y-k-r$ that is its codimension is smaller than $r$, i.e. $\rank~ess(X)<r$ and hence $\rank~X<r$. This implies that $\A'$ is non-very generic and hence $\A^0$ is non-very generic. 
\end{proof}
\noindent
Theorem \ref{thm:main2} allows to build non-very generic arrangements simply imposing linear conditions on vectors $v_{i,j} \in H^0_{i,j}$. This linearity is a non trivial achievement since the conditions to check the (non) very genericity are Pl\"ucker-type conditions. We point out that while Theorem \ref{thm:main2} provides a quite useful tool to build non-very generic arrangements, we are still far away from being able to check whether a given arrangement is very generic or not.\\ 
In the next section we will provide non trivial examples of how to build non-very generic arrangements by means of the Theorem \ref{thm:main2}. 

\section{Examples of non-very generic arrangements}\label{sec:example}
In this section we present few examples to illustrate how to use the Theorem \ref{thm:main2} to construct non-very generic arrangements. To construct the numerical examples we used the software CoCoA-5.2.4 ( see \cite{AM}).

\begin{ex}[$\B(12, 8, \A^0)$ with an intersection of multiplicity 4 in rank 3]\label{ex:MS(12,8)}
Let $L_1 = [12] \setminus \{ 10,11,12 \}, L_2 = [12] \setminus  \{ 7,8,9 \}, L_3 = [12] \setminus\{ 4,5,6 \}$ and $L_4 = [12] \setminus \{ 1,2,3 \}$ be subsets of $[12]$ of $k+1=9$ indices. It is an easy computation that the set $\TT = \{ L_1, L_2, L_3, L_4 \}$ is a $4$-set. 
Let's consider a central generic arrangement $\A^0$ of $12$ hyperplanes in $\CC^8$. In this case $m=n=12, y=0$ and $m-k-r=12-8-4=0$, hence, by Theorem \ref{thm:main2} in order for $\A^0$ to be non-very generic it is enough the existence of just one $K_{\TT}$-vector set $\{v_{1,2},v_{1,3},v_{1,4}\}$, that is the vectors $v_{2,3}=v_{1,3}-v_{1,2} \in \bigcap_{p \in L_2 \cap L_3 \setminus \{ 12 \}} H_p^0$ and $v_{2,4}=v_{1,4}-v_{1,2} \in \bigcap_{p \in L_2 \cap L_4 \setminus \{ 12 \}} H_p^0$ have to belong to $H_{12}^0$ (see Figure \ref{fig:exB(12,8)}). Notice that since $v_{3,4} = v_{2,4} - v_{2,3} \in \bigcap_{p \in L_3 \cap L_4 \setminus \{ 12 \}} H_p^0$, if $v_{2,3}, v_{2,4} \in H_{12}^0$ then $v_{3,4} \in H_{12}^0$. That is all hyperplanes in $\A^0$ can be chosen freely\footnote{Here and in the rest of this section, freely means that we only impose the condition that $\A^0$ is a central generic arrangement. In particular this condition is always taken as given and imposed even if not written.}, but $H_{12}$ which has to contain the vectors $v_{2,3},v_{2,4}$.

\begin{figure}[h]
\begin{minipage}{0.48\textwidth}
\centering
\begin{tikzpicture}
\coordinate (0) at (0,0);
\coordinate (1) at (4,0);
\coordinate (2) at (3,4);
\coordinate (3) at (1,3);

\coordinate (4) at (-1,0);
\coordinate (5) at (-1/2,-2/3);
\coordinate (6) at (-1/3,-1);
\coordinate (7) at (17/4,-1);
\coordinate (8) at (9/2,-1/2);
\coordinate (9) at (5,0);
\coordinate (10) at (4,9/2);
\coordinate (11) at (4,16/3);
\coordinate (12) at (11/4,5);
\coordinate (13) at (4/3,4);
\coordinate (14) at (0,4);
\coordinate (15) at (0,5/2);

\coordinate [label=above:$P_1$] (a) at (-1/4,0);
\coordinate [label=above:$P_2$] (b) at (17/4,0);
\coordinate [label=left:$P_3$] (c) at (3,25/6);
\coordinate [label=above:$P_4$] (d) at (1-1/9,3+1/9);
\coordinate [label=$H_{1,2}^t$] (e) at (-1.3,-0.3);
\coordinate [label=$H_{1,3}^t$] (f) at (-2/3-0.1,-1-0.1);
\coordinate [label=$H_{1,4}^t$] (g) at (-1/4-0.1,-1.6);
\coordinate [label=$\bigcap_{p \in L_2 \cap L_3 \setminus \{ 12 \}} H_p^t$] (h) at (4+0.3,-1.6);
\coordinate [label=$\bigcap_{p \in L_2 \cap L_4 \setminus \{ 12 \}} H_p^t$] (j) at (5.7,-1-0.1);
\coordinate [label=$\bigcap_{p \in L_3 \cap L_4 \setminus \{ 12 \}} H_p^t$] (i) at (5.3,25/6+0.2);

\begin{scope}
\draw[-latex] (0) -- node[below] {$v_{1,2}$} (1);
\draw[-latex] (1) -- node[right] {$v_{2,3}$} (2);
\draw[-latex] (2) -- node[above] {$v_{3,4}$} (3);
\draw[-latex] (0) -- node[left] {$v_{1,4}$} (3);
\draw[-latex] (0) -- node[left] {$v_{1,3}$} (2);
\draw[-latex] (1) -- node[right] {$v_{2,4}$} (3);
\draw (4) -- (9);
\draw (5) -- (11);
\draw (7) -- (12);
\draw (6) -- (13);
\draw[dashed] (8) -- (14);
\draw (10) -- (15);
\end{scope}
\end{tikzpicture}
\caption{$K_\TT$-configuration $K_\TT(\A^t)$ of $\B(12, 8, \A^0).$ $v_{i,j}$ are vectors in $H_{i,j}^0$.}\label{fig:exB(12,8)}
\end{minipage}
\begin{minipage}{0.48\textwidth}
\centering
\begin{tikzpicture}
\coordinate (0) at (-1, 0);
\coordinate (1) at (-1/2, -5/12);
\coordinate (2) at (-1/4,-1);
\coordinate (3) at (1/3,-5/6);
\coordinate (4) at (5/3,-5/6);
\coordinate (5) at (22/10, -4/5);
\coordinate (6) at (27/10, -7/12);
\coordinate (7) at (3,0);
\coordinate (8) at (7/2,17/8);
\coordinate (9) at (7/2,5/2);
\coordinate (10) at (7/2, 35/12);
\coordinate (11) at (33/10, 13/4);
\coordinate (12) at (3/2, 35/8);
\coordinate (13) at (12/10, 24/5);
\coordinate (14) at (8/10, 24/5);
\coordinate (15) at (1/2, 35/8);
\coordinate (16) at (-3/2,15/4);
\coordinate (17) at (-3/2,35/12);
\coordinate (18) at (-3/2, 5/2);
\coordinate (19) at (-3/2,17/8);

\coordinate [label=left:$P_1$] (a) at (0,0.2);
\coordinate [label=right:$P_2$] (b) at (2,0.2);
\coordinate [label=$P_3$] (c) at (3,5/2);
\coordinate [label=right:$P_4$] (d) at (1,4);
\coordinate [label=$P_5$] (e) at (-1,5/2);

\coordinate (A) at (0,0);
\coordinate (B) at (2,0);
\coordinate (C) at (3,5/2);
\coordinate (D) at (1,4);
\coordinate (E) at (-1,5/2);

\coordinate [label=$H_{1,2}^t$] (p) at (-1.3,-0.3);
\coordinate [label=$H_{1,3}^t$] (q) at (-2/3-0.1,-1);
\coordinate [label=$H_{1,4}^t$] (r) at (-1/4-0.1,-1.6);
\coordinate [label=$H_{1,5}^t$] (s) at (0.5, -1.4);
\coordinate [label=$H_{2,3}^t$] (t) at (1.7, -1.4);
\coordinate [label=$H_{2,4}^t$] (u) at (2.3, -1.4);
\coordinate [label=$H_{2,5}^t$] (v) at (3.0, -1.2);
\coordinate [label=$H_{3,4}^t$] (w) at (3.8, 1.6);
\coordinate (x) at (4.8, 2.2);
\coordinate (y) at (2.7, 4.3);

\begin{scope}
\draw (0) -- (7);
\draw (1) -- (10);
\draw (2) -- (13);
\draw (3) -- (16);
\draw (4) -- (11);
\draw (5) -- (14);
\draw (6) -- (17);
\draw (8) -- (15);
\draw [dashed] (9) -- (18);
\draw [dashed] (12) -- (19);

\draw[-latex] (A) -- node {$v_{1,2}^{k}$} (B);
\draw[-latex] (A) -- node[left] {$v_{1,3}^{k}$} (C);
\draw[-latex] (A) -- node {$v_{1,4}^{k}$} (D);
\draw[-latex] (A) -- node {$v_{1,5}^{k}$} (E);
\draw[-latex] (B) -- (C);
\draw[-latex] (B) -- (D);
\draw[-latex] (B) -- (E);
\draw[-latex] (C) -- (D);
\draw[-latex] (C) -- (E);
\draw[-latex] (D) -- (E);

\end{scope}
\end{tikzpicture}

\caption{$K_\TT$-configuration $K_\TT(\A^t)$ of $\B(10, 3, \A^0).$ $v_{i,j}^{k}$ is a vector in $H_{i,j}^0$.}\label{fig:exB(10,3)}
\end{minipage}
\end{figure}

\noindent
Let's see a numerical example. Let us consider hyperplanes of equation $H_i^0: \alpha_i \cdot x = 0$, with $\alpha_i$, $i = 1, \dots, 11$ assigned as following: 
\begin{equation}
\begin{split}
& \alpha_1 = (0,0,1,1,0,1,-1,1), \alpha_2 = (0,0,0,1,1,1,1,-1), \alpha_3 = (0,0,1,0,0,0,1,1), \\
& \alpha_4 = (0,1,0,1,1,1,0,1), \alpha_5 = (0,2,0,-1,-1,0,1,-1), \alpha_6 = (0,-1,0,2,1,-1,-1,1), \\
& \alpha_7 = (1,0,0,1,0,-1,-1,1), \alpha_8 = (-1,0,0,0,2,1,1,1), \alpha_9 = (-4,0,0,0,1,-1,1,1), \\
& \alpha_{10} = (1,1,1,-1,-1,-1,-1,1), \alpha_{11} = (1,1,1,2,2,2,0,3). \\
\end{split}
\end{equation}
In this case, we have the $K_\TT$-vector set
\begin{equation*}
\{ v_{1,2}, v_{1,3}, v_{1,4} \} = \{ (1,0,0,0,0,0,0,0), (0,1,0,0,0,0,0,0), (0,0,-1,0,0,0,0,0) \} \quad .
\end{equation*}
The other vectors are obtained by means of relations $v_{2,3} = v_{1,3} - v_{1,2}, v_{2,4} = v_{1,4} - v_{1,2}, v_{3,4} = v_{1,4} - v_{1,3}$, that is
\begin{equation}
v_{2,3} = (-1,1,0,0,0,0,0,0), v_{2,4} = (-1,0,1,0,0,0,0,0), v_{3,4} = (0,-1,1,0,0,0,0,0) 
\end{equation}
and, finally, we get  $\alpha_{12} = (-2,-2,-2,3,4,-5,6,7)$ by imposing the condition that $\alpha_{12}$ has to be orthogonal to $v_{2,3}$ and $v_{2,4}$ . 
\end{ex}

\begin{ex}[$\B(16, 11, \A^0)$ with an intersection of multiplicity 4 in rank 3]\label{ex:MS(16,11)}
Let $L_1 = [16] \setminus \{ 13,14,15,16 \}, L_2 = [16] \setminus \{ 9,10,11,12 \}, L_3 = [16] \setminus \{ 5,6,7,8 \}$ and $L_4 = [16] \setminus \{ 1,2,3,4 \}$ be subsets of $[16]$ of $k+1 = 12$ indices. The set $\TT = \{ L_1, L_2, L_3, L_4 \}$ is a $4$-set. Let's consider a central generic arrangement $\A^0$ of 16 hyperplanes in $\CC^{11}$. In this case $m = n = 16, y=0$ and $m - k -  r = 16 - 11 - 4 = 1$, hence, by Theorem \ref{thm:main2} in order for $\A^0$ to be non-very generic we need two weakly linearly independent $K_\TT$-vector sets $\{ v_{1,2}^{1}, v_{1,3}^{1}, v_{1,4}^{1} \}$ and $\{ v_{1,2}^{2}, v_{1,3}^{2}, v_{1,4}^{2} \}$ that is the vectors $v_{2,3}^{k} \in \bigcap_{p \in L_2 \cap L_3 \setminus \{ 16 \}} H_p^0$ and $v_{2,4}^{k} \in \bigcap_{p \in L_2 \cap L_4 \setminus \{ 16 \}} H_p^0$, $k = 1,2$, have to belong to $H_{16}^0$\footnote{The graphic representation in this case can be simply obtained replacing the number $12$ with $16$ in Figure \ref{fig:exB(12,8)}.}. Notice that since $v_{3,4}^{k} = v_{2,4}^{k} - v_{2,3}^{k} \in \bigcap_{p \in L_2 \cap L_4 \setminus \{ 16 \}} H_p^0$,  if $v_{2,3}^{k}, v_{2,4}^{k} \in H_{16}^0$ then $v_{3,4}^{k} \in H_{16}^0$. That is all hyperplanes in $\A^0$ can be chosen freely, but $H_{16}^0$ which has to contain the vectors $v_{2,3}^{k}, v_{2,4}^{k}$, $k = 1,2$.\\
Let's see a numerical example. Let us consider hyperplanes of equation $H_i^0: \alpha_i \cdot x = 0$, with $\alpha_i$, $i = 1, \dots, 15$ assigned as following. 
\begin{equation}
\begin{split}
& \alpha_1 = (0,0,1,0,0,1,0,0,0,1,-1), \alpha_2 = (0,0,-1,0,0,1,1,1,1,-1,0), \alpha_3 = (0,0,2,0,0,1,1,0,1,1,0), \\
& \alpha_4 = (0,0,1,0,0,1,1,0,0,0,1), \alpha_5 = (0,-1,0,0,1,0,1,1,1,-1,0), \alpha_6 = (0,1,0,0,2,0,0,-1,-1,0,1), \\
& \alpha_7 = (0,2,0,0,-1,0,-1,0,0,1,1), \alpha_8 = (0,-1,0,0,2,0,1,1,1,0,0), \alpha_9 = (1,0,0,-3,0,0,-1,-1,1,1,1), \\
& \alpha_{10} = (2,0,0,5,0,0,1,-1,-1,1,1), \alpha_{11} = (3,0,0,1,0,0,1,-1,2,0,1), \alpha_{12} = (1,0,0,5,0,0,1,0,1,1,0), \\
& \alpha_{13} = (1,1,1,-3,-3,-3,-1,-3,2,-2,-1), \alpha_{14} = (1,1,1,0,0,0,-2,1,-8,1,1), \alpha_{15} = (0,0,0,-5,-5,-5,1,2,-3,-4,7). \\
\end{split}
\end{equation}
In this case, we have the $K_\TT$-vector sets
\begin{equation*}
\begin{split}
& \{ v_{1,2}^{1}, v_{1,3}^{1}, v_{1,4}^{1} \} = \{ (1,0,0,0,0,0,0,0,0,0,0), (0,1,0,0,0,0,0,0,0,0,0), (0,0,1,0,0,0,0,0,0,0,0) \} \quad ,  \\
& \{ v_{1,2}^{2}, v_{1,3}^{2}, v_{1,4}^{2} \} = \{ (0,0,0,1,0,0,0,0,0,0,0), (0,0,0,0,1,0,0,0,0,0,0), (0,0,0,0,0,1,0,0,0,0,0) \} \quad .
\end{split}
\end{equation*}
The other vectors are obtained by means of relations $v_{2,3}^{k} = v_{1,3}^{k} - v_{1,2}^{k}, v_{2,4}^{k} = v_{1,4}^{k} - v_{1,2}^{k}, v_{3,4}^{k} = v_{1,4}^{k} - v_{1,3}^{k}$, $k = 1,2$, that is
\begin{equation}
\begin{split}
& v_{2,3}^{1} = (-1,1,0,0,0,0,0,0,0,0,0), v_{2,4}^{1} = (-1,0,1,0,0,0,0,0,0,0,0), v_{3,4}^{1} = (0,-1,1,0,0,0,0,0,0,0,0) \quad , \\
& v_{2,3}^{2} = (0,0,0,-1,1,0,0,0,0,0,0), v_{2,4}^{2} = (0,0,0,-1,0,1,0,0,0,0,0), v_{3,4}^{2} = (0,0,0,0,-1,1,0,0,0,0,0) 
\end{split}
\end{equation}
and, finally, we get $\alpha_{16}$ = $(1,1,1,-2,-2,-2,5,6,7,8,9)$ by imposing the conditions that $\alpha_{16}$ has to be orthogonal to $v_{2,3}^{k}$ and $v_{2,4}^{k}$, $k = 1,2$. 
\end{ex}

\begin{ex}[$\B(10, 3, \A^0)$ with an intersection of multiplicity 5 in rank 4]\label{ex:MS(10,3)}
Let $L_1 = \{ 1,2,3,4 \}, L_2 = \{ 1,5,6,7 \}, L_3 = \{ 2,5,8,9 \}, L_4 = \{ 3,6,8,10 \}$ and $L_5 = \{ 4,7,9,10 \}$ be subsets of $[10]$ of $k+1 = 4$ indices. The set $\TT = \{ L_1, L_2, L_3, L_4, L_5 \}$ is a $5$-set. Let's consider a central generic arrangement $\A^0$ of 10 hyperplanes in $\CC^3$. In this case $m = n = 10,y=0$ and $m - k -  r = 10 - 3 - 5 = 2$, hence, by Theorem \ref{thm:main2} in order for $\A^0$ to be non-very generic we need three weakly linearly independent $K_\TT$-vector sets $\{ v_{1,2}^{1}, v_{1,3}^{1}, v_{1,4}^{1}, v_{1,5}^{1} \}$, $\{ v_{1,2}^{2}, v_{1,3}^{2}, v_{1,4}^{2}, v_{1,5}^{2} \}$ and $\{ v_{1,2}^{3}, v_{1,3}^{3}, v_{1,4}^{3}, v_{1,5}^{3} \}$ that is the vectors $v_{4,5}^{k}$, $k = 1,2,3$, have to belong to $H_{10}^0$  (see Figure \ref{fig:exB(10,3)}). Notice that since in this case hyperplanes are planes, then the three vectors $v_{i,j}^{k}, k=1,2,3$ will be linearly dependent for any choice of indices $(i,j),i\neq j$. This additional condition forces that at most 8 hyperplanes in $\A^0$ can be chosen freely, while both $H_{9}^0$ and $H_{10}^0$ have to contain the dependent vectors $v_{3,5}^{k}$ and $v_{4,5}^{k}$, $k = 1,2,3$, respectively.\\
Let's see a numerical example. Let us consider hyperplanes of equation $H_i^0: \alpha_i \cdot x = 0$, with $\alpha_i$, $i = 1, \dots, 8$ assigned as following. 
\begin{equation}
\begin{split}
& \alpha_1 = (0,10,3), \alpha_2 = (20,0,-9), \alpha_3 = (2,-3,0), \alpha_4 = (3,1,0), \\
&\alpha_5 = (0,0,1), \alpha_6 = (1,-1,1), \alpha_7 = (1,2,2), \alpha_8 = (4,-1,-3).
\end{split}
\end{equation}
In this case, we have the $K_\TT$-vector sets
\begin{equation*}
\begin{split}
& \{ v_{1,2}^{1}, v_{1,3}^{1}, v_{1,4}^{1}, v_{1,5}^{1} \} = \{ (1,-3,10), (\frac{9}{2},\frac{21}{2},10),(\frac{9}{2},3,\frac{25}{2}),(-\frac{77}{9},\frac{77}{3},-\frac{125}{9}) \}, \\
& \{ v_{1,2}^{2}, v_{1,3}^{2}, v_{1,4}^{2}, v_{1,5}^{2} \} = \{ (-2,6,-20),(-9,-47,-20),(-3,-2,-27),(-\frac{2}{3},2,-\frac{50}{3}) \}, \\
& \{ v_{1,2}^{3}, v_{1,3}^{3}, v_{1,4}^{3}, v_{1,5}^{3} \} = \{ (-3,3,-10),(-\frac{9}{2},-\frac{2391}{80},-10),(-\frac{1467}{1040},-\frac{489}{520},-\frac{16151}{1040}),(-\frac{4}{3},4,-\frac{71}{6}) \}.
\end{split}
\end{equation*}
The other vectors are obtained by means of relations $v_{i,l}^{k} = v_{1,l}^{k} - v_{1,i}^{k}$, where $2 \leq i<l \leq 5$, $k = 1,2,3$, that is
\begin{equation}
\begin{split}
& v_{2,3}^{1} = (\frac{7}{2},  \frac{27}{2},  0), v_{2,4}^{1} = (\frac{7}{2},  6,  \frac{5}{2}), v_{2,5}^{1} = (-\frac{86}{9},  \frac{86}{3},  -\frac{215}{9}), \\
& v_{3,4}^{1} = (0,  -\frac{15}{2},  \frac{5}{2}), v_{3,5}^{1} = (-\frac{235}{18},  \frac{91}{6},  -\frac{215}{9}), v_{4,5}^{1} = (-\frac{235}{18},  \frac{68}{3},  -\frac{475}{18}), \\ 
& v_{2,3}^{2} = (-7,  -53,  0), v_{2,4}^{2} = (-1,  -8,  -7), v_{2,5}^{2} = (\frac{4}{3},  -4,  \frac{10}{3}), \\
& v_{3,4}^{2} = (6,  45,  -7), v_{3,5}^{2} = (\frac{25}{3},  49,  \frac{10}{3}), v_{4,5}^{2} = (\frac{7}{3},  4,  \frac{31}{3}), \\
& v_{2,3}^{3} = (-\frac{3}{2},  -\frac{2631}{80},  0), v_{2,4}^{3} = (\frac{1653}{1040},  -\frac{2049}{520},  -\frac{5751}{1040}), v_{2,5}^{3} = (\frac{5}{3},  1,  -\frac{11}{6}), \\
& v_{3,4}^{3} = (\frac{3213}{1040},  \frac{6021}{208},  -\frac{5751}{1040}), v_{3,5}^{3} = (\frac{19}{6},  \frac{2711}{80},  -\frac{11}{6}), v_{4,5}^{3} = (\frac{241}{3120},  \frac{2569}{520},  \frac{11533}{3120}) \quad .
\end{split}
\end{equation}
Finally, we get $\alpha_9$ = $(314,-40,-197)$ and $\alpha_{10}$ = $(139,30,-43)$ by imposing the conditions that $\alpha_9$ and $\alpha_{10}$ have to be orthogonal to $v_{3,5}^{k}$ and $v_{4,5}^{k}$, $k = 1,2,3$. 
\end{ex}

\begin{rem}
Notice that Example \ref{ex:MS(10,3)} is slightly different from other examples for two reasons. 
Firslty, it uses a different combinatorics. In the Examples \ref{ex:MS(12,8)} and \ref{ex:MS(16,11)} the 4-sets $\TT = \{ L_1, L_2, L_3, L_4 \}$ are of the form $L_i = [n] \setminus K_i$ with $K_i$'s which satisfy the properties $\bigcup_{i=1}^4 K_i=[n]$ and $K_i \cap K_j = \emptyset$\footnote{Notice that this is a generalization of the combinatorics used in \cite{LS}} while in the Example \ref{ex:MS(10,3)} they are not. Secondly, in the Examples \ref{ex:MS(12,8)} and \ref{ex:MS(16,11)} in order to obtain non-very generic arrangement we could choose all hyperplanes freely but one, while in the Example \ref{ex:MS(10,3)} two hyperplanes had to be fixed as a result of the need of three weakly independent $K_\TT$-vector sets in two dimensional hyperplanes. Indeed this dependency condition gives rise to $27$ independent equations of the form 
\begin{equation}\label{eq:rem_dep2}
v_{i,j}^3 = \alpha v_{i,j}^1 + \beta v_{i,j}^2
\end{equation}
which fix the entries of the vectors $v_{1,i}^k, i = 3,4,5$ uniquely for any choice of three dependent vectors $v_{1,2}^k, k = 1,2,3$. Hence the vectors $v_{3,5}^k$ and $v_{4,5}^k$, $k = 1,2,3$ are determined and so are the two hyperplanes $H_9^0$ and $H_{10}^0$. 
\end{rem}

\begin{rem} Notice that both Example \ref{ex:MS(16,11)} and Example \ref{ex:MS(10,3)} satisfy Athanasiadis condition while the first one fails when the set $I$ has maximal cardinality. This essentially shows how the problem to describe the $r$-sets $\TT$ that can give rise to ( simple ) non-very generic intersections is non trivial. 
\end{rem}

\bigskip 

\noindent
\textbf{Competing interests:} The author(s) declare none .


\end{document}